\documentclass[12pt,a4paper]{amsart}
\usepackage{graphics}
\usepackage{epsfig}
\usepackage{graphicx}
\theoremstyle{plain}
\usepackage{amssymb}
\advance\hoffset-20mm \advance\textwidth40mm

\newtheorem{theorem}{Theorem}
\newtheorem{lemma}{Lemma}
\newtheorem*{theo*}{Theorem}
\newtheorem{proposition}{Proposition}
\newtheorem{corollary}{Corollary}
\theoremstyle{definition}
\newtheorem{definition}{Definition}
\newtheorem*{definition*}{Definition}
\newtheorem{example}{Example}
\newtheorem{remark}{Remark}

\begin{document}
\sloppy
\title[On finite-dimensional subalgebras of derivation algebras]{On finite-dimensional subalgebras of derivation algebras}


%
\begin{abstract}
Let $\mathbb{K}$ be a field, $R$ be an associative and commutative
$\mathbb{K}$-algebra and $L$ be a Lie algebra over $\mathbb{K}$.
We give some descriptions of injections from $L$ to Lie
algebra of $\mathbb{K}$-derivations of $R$ in the terms
of $L$.
\end{abstract}
\maketitle


\section{Introduction}

Let $\mathbb{K}$ be a field, $R$ be an associative and commutative
$\mathbb{K}$-algebra and $L$ be a Lie algebra over $\mathbb{K}$.
Let $Der_{\mathbb{K}}R$ be a Lie algebra of 
$\mathbb{K}$-derivations of the algebra $R$. 
We study subalgebras from $Der_{\mathbb{K}}R$. This
questions are studied in many papers (see \cite{Lie}, \cite{Now}, \cite{Arzh}, \cite{PBNL}).
Jan Draisma in his PhD. thesis (\cite{Dra}) studied so called transitive derivation
of $\mathbb{K}[x_1, \dots, x_n]$.

More precisely we study injective morphisms $\varphi$ from $L$ to $Der_{\mathbb{K}}R$.
$Der_{\mathbb{K}}R$ has a structure of $R$-module.
Put $M=R\varphi(L)\subset Der_{\mathbb{K}}R$. If $R$ is a field then
$M$ is a vector space over $R$. Put $k=Dim_R M$. We prove that if such injection exists then
$R\otimes L$ has an $R$-subalgebra of codimension $k$ and if $L$ is finite dimensional then $L$ has a
subalgebra of codimension $k$.

Given a Lie algebra $L$ and $k \in \mathbb{N}$.
We want to give an answer in terms of $L$ where there exist
a field $R$ and an injection $\varphi:L \rightarrow Der_{\mathbb{K}}R$
such that $Dim_R R \varphi{L}=k$.
Let $M_k(L)$ be an affine algebraic set of all elements $(l_1, \dots l_k)\in L^k$
such that a vector space $\mathbb{K}l_1+ \dots +\mathbb{K}l_k $ is a subalgebra
and $M^0_k(L)$ be a subset of $M_k(L)$ such that an elements $l_1, \dots l_k$ are linear dependent.
 The main result of this paper is the following Theorem.
 For a closed point $p$ of an affine algebraic variety we denote an evaluation of element $x$ from coordinate ring
 in $p$ by $p(x)$.

Theorem 1. Let $L$ be a finite dimensional Lie algebra over an algebraically closed field $\mathbb{K}$.
Then the following conditions are equivalent:

$(i)$ there exists a field $R$ and an injective homomorphism $\varphi: L \rightarrow Der_{\mathbb{K}}R$
such that $Dim_R R \varphi(L)=k$;

$(ii)$ there exist an irreducible subvariety of $V \subset M_k(L)$,
$V \nsubseteq M^0_k(L)$, such that an intersection of subalgebras associated
to points of $V$ is equal to $\{0\}$.

A technique of this paper gives some algorithms constructing injections. They
are used in examples.

\section{Construction of some Lie algebras}

Let $\mathbb{K}$ be a field, $R$ be associative and commutative $\mathbb{K}$-algebra
and $L$ be a Lie algebra over $\mathbb{K}$.
Consider a Lie algebra $Der_{\mathbb{K}}R$. Let $M$ be an $R$-submodule of $Der_{\mathbb{K}}R$
closed under Lie bracket and $\varphi: L \rightarrow M$ be an injective morphism of Lie algebras.
Consider a Lie algebra $R \otimes L$ which is a free $R$-module. For an arbitrary derivation
$d \in Der_{\mathbb{K}}R$ denote a derivation of the Lie algebra  $d'\in Der_{\mathbb{K}}(R \otimes L)$:
$d'(m \otimes l)=d(m) \otimes l$. Note that a mapping $d \rightarrow d'$ is a
monomorphism of Lie algebras. Thus we can define a following Lie algebra.

\begin{definition}
Let $L$ be a Lie algebra over a field $\mathbb{K}$, $R$ be an associative and commutative $\mathbb{K}$-algebra
and $M$ be an $R$-submodule of $Der_{\mathbb{K}}R$. Define the following Lie algebra
\[\mathfrak{D}=\mathfrak{D}(L, R, M)=M \rightthreetimes (R \otimes L).\]

\end{definition}

\begin{remark}
In the case $R=\mathbb{K}[x_1, \dots, x_n]$ $\mathfrak{D}$ is a polynomial Lie algebra that are
defined in \cite{Buh}.
\end{remark}

Denote an action of this algebra on the ring $R$: $(D + r \otimes l)(s)=D(s)$ for $D \in M$,
$l \in L$, $r,s \in R$.

\begin{lemma}
Let $d_1, d_2$ be an elements of $\mathfrak{D}$, $r,s$ be an elements of $R$. Then:
\[[r d_1, s d_2]=r d_1(s) d_2 -s d_2(r) d_1 +rs [d_1, d_2].\]
\end{lemma}

\begin{proof}
Straightforward calculation.
\end{proof}

Let $\lbrace l_i \rbrace, i \in I$ be a basis of $L$. Denote a following monomorphism  from $L$
to $\mathfrak{D}$: $\varphi': l_i \mapsto \varphi (l_i) + 1 \otimes l_i, i \in I$. Note that
$\varphi'(L)$ is a Lie algebra isomorphic to $L$.

\begin{definition}
$\widetilde{L}= R\varphi'(L) \cap R \otimes L$.
\end{definition}

It is easy to see that $\widetilde{L}$ is a subalgebra of a Lie algebra $R \otimes L$
considered as an algebra \emph{over $R$}.

\begin{lemma}\label{elementsofsubalgebra}
Let $r_1, \dots , r_p$ be an elements of $R$. Then $r_1 \varphi (l_{i_1})+ \dots + r_p \varphi (l_{i_p})=0$
if and only if $r_1 \otimes l_{i_1} + \dots + r_p \otimes l_{i_p} \in \widetilde{L}$.
\end{lemma}

\begin{proof}
Let $r_1 \varphi (l_{i_1})+ \dots + r_p \varphi (l_{i_p})=0$.
Then:
\[r_1 (\varphi (l_{i_1}) + 1 \otimes l_{i_1})+ \dots + r_p (\varphi (l_{i_p}) + 1 \otimes l_{i_p})=
 r_1 \otimes l_{i_1} + \dots + r_p \otimes l_{i_p} \in R\varphi'(L) \cap R \otimes L.\]

 Conversely let $r_1 \otimes l_{i_1} + \dots + r_p \otimes l_{i_p} \in \widetilde{L}$.
 Then:
 \[r_1 \otimes l_{i_1} + \dots + r_p \otimes l_{i_p}=\sum_{i \in I}(s_i\varphi(l_i) + s_i \otimes l_i),\]
 where only finite number of the elements $s_i$ are nonzero.
 But $\mathfrak{D} = M \oplus R \otimes L$ as $R$-module. Therefore $\sum_{i \in I}s_i\varphi(l_i) =0$
 and $\sum_{i \in I} s_i \otimes l_i=r_1 \otimes l_{i_1} + \dots + r_p \otimes l_{i_p}$.
 A set $\lbrace l_i \rbrace$ is a free basis of $R \otimes L$. Thus $r_j=s_{i_j}, j=1 \dots p$
 and all other $s_i$'s are equal to $0$. This completes the proof of Lemma.

\end{proof}

\begin{lemma}\label{intersection}
$\widetilde{L} \cap 1 \otimes L = \lbrace 0 \rbrace$.
\end{lemma}
\begin{proof}
  Assume that $0 \neq 1 \otimes l \in \widetilde{L} \cap 1 \otimes L$, i. e.
$1 \otimes l \in R \varphi'(L)$. Therefore for some nonzero $r_1, \dots, r_m \in R$ we have:
\[l=r_1(\varphi(l_{i_1})+ 1 \otimes l_{i_1})+ \dots + r_m(\varphi(l_{i_m})+ 1 \otimes l_{i_m}).\]
Therefore $r_1\varphi(l_{i_1})+ \dots + r_m\varphi(l_{i_m})=0$ and
$r_1 \otimes l_{i_1}+ \dots + r_m \otimes l_{i_m} =l$. Hence $r_j \in \mathbb{K}$, $j= 1, \dots, m$.
Thus $\varphi(r_1 l_{i_1} + \dots + r_m l_{i_m})=0$ with contradiction to injectivity of $\varphi$.
\end{proof}

 \section{Subalgebras of codimension $k$}
During this chapter $R$ will be a field.

\begin{proposition}
Let $L$ be a Lie algebra over $\mathbb{K}$, $\varphi: L \rightarrow Der_{\mathbb{K}}R$ be an injection,
$M=R\varphi(L)$ be a vector space of dimension $k$ over $R$. Then $\widetilde{L}$ is a
subalgebra of $R$-algebra $R \otimes L$ with codimension $k$.
\end{proposition}
\begin{proof}
Let $\lbrace l_1, \dots, l_k \rbrace\subset L$ be an elements of $L$ such that
 $\lbrace \varphi(l_1), \dots, \varphi(l_k) \rbrace$ form a basis of $M$. It is easy to see that
 $M$ is closed under multiplication.

 For any element of $l \in L$ and some elements $r_1, \dots, r_k$ we have:
 \[\varphi(l)+r_1 \varphi(l_1) + \dots + r_k\varphi(l_k)=0. \]
 Therefore using Lemma \ref{elementsofsubalgebra} we have that $l+r_1 l_1 + \dots + r_k l_k \in \widetilde{L}$.
 Also we have that $\lbrace \varphi(l_1), \dots, \varphi(l_k) \rbrace$ are linear independent over $R$.
 Hence using Lemma \ref{elementsofsubalgebra} we have that for any $r_1, \dots, r_k \in R$
 $r_1 l_1 + \dots + r_k l_k \notin \widetilde{L}$. Thus $\widetilde{L}$ has codimension $k$ in
 $R \otimes L$.
\end{proof}

\begin{lemma}\label{injectionconditions}
Let $R$ be a field, $L$ be a finite dimensional Lie algebra over $\mathbb{K}$. Then there exist an injection
$\varphi: L \rightarrow Der_{\mathbb{K}}R$ such that $Dim_R R \varphi(L)=k$ if and only if
there exist a subalgebra $\widehat{L} \subset R \otimes L$ of $R$-codimension $k$,
$\widehat{L} \cap 1 \otimes L=\{0\}$, a
set of elements $\bar{b}_1, \dots, \bar{b}_k \in R \otimes L$, $\widehat{L} + R \bar{b}_1 + \dots + R\bar{b}_k=R \otimes L$
and a set of linear independent over $R$ derivations $D_1, \dots D_k$ such that an $R$-module
$R(D_1 + \bar{b}_1)+ \dots + R(D_k + \bar{b}_k)+ \widehat{L}$ is closed under Lie bracket.
\end{lemma}

\begin{proof}
Assume that there exist a needed injection $L \rightarrow Der_{\mathbb{K}}R$.
Consider a Lie algebra $R\varphi'(L)$. Take an elements $l_1, \dots , l_k$ such
that $\varphi(l_i), i=1, \dots k$ are linearly independent over $\mathbb{K}$.
Then $R\varphi'(L)=R(\varphi(l_1)+ 1 \otimes l_1)+ \dots + R(\varphi(l_k)+ 1 \otimes l_k) + \widetilde{L}$.
Put $\widehat{L}=\widetilde{L}$. Assume that $r_1\otimes l_1 + \dots + r_k \otimes l_k \in \widetilde{L}$
for some $r_1, \dots, r_k \in R$. Then using Lemma \ref{elementsofsubalgebra} we have that
$r_1 \varphi(l_1) + \dots + r_k \varphi(l_k)=0$ with contradiction to linear independency of $\{\varphi(l_i), i=1, \dots, k\}$.
Therefore $\widehat{L} + R \otimes l_1 + \dots + R \otimes l_k=R \otimes L$.

Conversely, consider an $R$-basis of subalgebra $\widehat{L}$, $\widehat{L}=
\sum_{i=k+1}^m R \sum_{j=1}^m b_{ij} \otimes l_j$, $b_{ij} \in R$.
Put $\bar{b}_i=\sum _{j=1}^m b_{ij} \otimes l_j, i=i, \dots, k$.
Put $M=RD_1+ \dots + RD_k$. It is easy to see that
$M\simeq R(D_1 + \bar{b}_1)+ \dots + R(D_k + \bar{b}_k)+ \widehat{L}/\left(
R(D_1 + \bar{b}_1)+ \dots + R(D_k + \bar{b}_k)+ \widehat{L}\cap R \otimes L \right)$ is closed under Lie bracket.
Consider the following linear system:
\begin{equation}
\sum_{j=1}^m b_{ij} d_j=D_i, i=1 \dots k;
\sum_{j=1}^m b_{ij} d_j=0, i=k+1, \dots m.\label{systemofequations}
\end{equation}

It is easy to see that this system has a unique solution. Denote a linear map $\varphi$
from $L$ to $M$ on basis as $\varphi(l_i)=d_i, i=1, \dots, m$.
For this map we have that $\{\varphi(l_i) + l_i\}$
is a basis of $R$-module $R(D_1 + \bar{b}_1)+ \dots + R(D_k + \bar{b}_k)+ \widehat{L}$.
This module is closed under Lie bracket. Therefore we obtain:
 \[[\varphi(l_i) + l_i, \varphi(l_j) + l_j]=\sum_{k=1}^m c_{ij}^s (\varphi(l_s) + l_s).\]
 Hence:
 \[[\varphi(l_i), \varphi(l_j)]=\sum_{s=1}^m c_{ij}^s \varphi(l_s);\]
 \[[l_i, l_j]=\sum_{k=1}^m c_{ij}^k l_k.\]
 Thus we have that $ c_{ij}^s \in \mathbb{K}$, $i,j,s =1, \dots, m$. Therefore
 a $\mathbb{K}$-space generated by $\varphi(l_i), i =1, \dots m$ is closed under Lie bracket
 and a linear map defined on a basis as $l_i \rightarrow \varphi(l_i)$ is
 a morphism of Lie algebras. This completes the proof of Lemma.
\end{proof}

\begin{remark}\label{noninjective}
Analogously to proof of the previous Lemma we obtain that in conditions of this Lemma
without condition $\widehat{L} \cap 1 \otimes L$ we have an analogous morphism
$\varphi: L \rightarrow Der_{\mathbb{K}} R$ that is not injective.
\end{remark}

Consider a following algebraic variety.
\begin{definition}
Let $L$ be a finitely dimension Lie algebras, $Dim L=m$. Consider a following variety:
\begin{equation} \label{varietysubalgebra}
M_k{L}=\lbrace (a_{ij}), i=1 \dots m-k, j=1 \dots m | \sum_{i=1}^{m-k}\mathbb{K}\sum_{j=1}^m a_{ij} l_j
\rm{is ~a ~subalgebra}. \rbrace
\end{equation}

It is easy to see that it is an affine algebraic variety. We will denote it as a variety of subalgebra bases with
codimension $k$. Consider a following closed subspace of $M_k{L}$.

\begin{equation}
M_k^0{L}=\lbrace (a_{ij}), i=1 \dots m-k, j=1 \dots m |  \sum_{i=1}^{m-k}\mathbb{K}\sum_{j=1}^m a_{ij} l_j
\rm{is ~a ~subalgebra}
\end{equation}
\[
\rm{and} \it \lbrace ~\sum_{j=1}^{m} a_{ij}l_j \rbrace \rm{ ~ is ~linear ~dependent ~set }. \rbrace
\]

\end{definition}

For any point $(a_{ij})$ of $M_k(L)$ consider an associated subalgebra $ \sum_{i=1}^{m-k}\mathbb{K}\sum_{j=1}^m a_{ij} l_j$.
It has codimension $k$ if and only if $(a_{ij}) \notin M^0_k(L)$.

\begin{proposition} \label{subalgebratovariety}
Let $\mathbb{K}$ be an algebraically closed field.
Let $L$ be a finite dimensional Lie algebra over the field $\mathbb{K}$.
Assume that there exist an injection $\varphi: L \rightarrow Der_{\mathbb{K}} R$,
$k=Dim_R R \varphi(L)$. Then:

$(i)$ $L$ has a subalgebra of codimension $k$;

$(ii)$ there exist an irreducible subvariety $V$ of $M_k(L)$ such that $V \nsubseteq M^0_k(L)$ and
intersection of all associated subalgebras to points of $V$ is $\{0\}$.

\end{proposition}

\begin{proof}
Put $m=DimL$
Consider a subalgebra $\widetilde{L}$. It has codimension $k$. Let $\lbrace \sum_{j=1}^m z_{ij} l_j \rbrace$,
$i=1, \dots, m$ be a basis of $\widetilde{L}$. Let $S$ be a  ring generated by $\{z_{ij}\}$,
in is a subring of field and therefore integral domain.
Then we have an obvious injection $\tau: Spec S \rightarrow M_k(L)$. Assume that
all subalgebras associated to points of $\tau(Spec S)$ have codimension greater then $k$. Then
all maximal minors of $(p(z_{ij}))$ are equal to zero for any evaluation in point $p$. Hence using
Hilbert's Nullstellensatz (see, for example, \cite{Hartshorne}, p. 4) we have that all maximal minors of $(z_{ij})$ are equal to $0$.
Thus $\widetilde{L}$ has codimension greater then $k$. This completes the proof of part $(i)$.

Assume that for any closed point $p$ we have that $\langle \sum_{j=1}^m p(z_{ij})\otimes l_j \rangle, i=1 \dots m-k$
has an element $\sum_{j=1}^m \alpha_j\otimes l_j$, $\alpha_j \in \mathbb{K}$.
Then the set $\{\sum_{j=1}^m p(z_{ij})\otimes l_j\}\cup \{\sum_{j=1}^m \alpha_j \otimes l_j\}$
is linear dependent.
Analogously to the proof of $(i)$ we obtain that a set $\{\sum_{j=1}^m z_{ij}l_j\}\cup \{\sum_{j=1}^m \alpha_j \otimes l_j\}$
is linear dependent. Thus $\sum_{j=1}^m \alpha_j \otimes l_j$ is an element of $\widetilde{L}$ with
contradiction to Lemma \ref{intersection}. Thus completes the proof of part $(ii)$.
\end{proof}

\begin{lemma} \label{subalgtoderivations}
Let $\mathbb{K}$ be a field, $R$ be a field containing $\mathbb{K}$,
$L$ be a finite dimensional Lie algebra over $\mathbb{K}$. Assume
that there exist an $R$-subalgebra $\widehat{L} \subset R \otimes L$
of codimension $k$.
Then there exist a monomorphism $\varphi:L \rightarrow Der_{\mathbb{K}}F$,
where $F$ is a quotient field of $R[[x_1, \dots, x_k]]$ such that
$Dim_F F\varphi(L) =k$. If $L$ is nilpotent then $F$ can
be changed to $\mathbb{K}(x_1, \dots, x_k)$.
\end{lemma}
\begin{proof}
Let $l_1, \dots, l_k$ be an elements of $L$ such that
$R \otimes l_1 + \dots + R \otimes l_k + \widehat{L}= R \otimes L$.
Put $M=F \frac{\partial}{\partial x_1} + \dots +=F \frac{\partial}{\partial x_k}$,
$\mathfrak{D}=M \rightthreetimes F \otimes L$. Then we have a natural
embedding $R \otimes L \rightarrow F \otimes L$. Let $\bar b_{k+1}, \dots, \bar b_{m}$
be a basis of $\widehat{L}$. Then it is a basis of
$F \otimes \widehat{L}$. Put $w=x_1 \otimes l_1 + \dots + x_k \otimes l_k$.
$N=F \frac{\partial}{\partial x_1} + \dots + F\frac{\partial}{\partial x_k}+ F \widehat{L}$
is closed under Lie bracket. Then an $F$-module $e^{ad w}N$ is closed under Lie bracket.
It has a basis $\lbrace e^{ad w} \frac{\partial}{\partial x_1},
e^ {ad w} \frac{\partial}{\partial x_k}, e^{ad w} \bar{b}_{k+1}, \dots, , e^{ad w} \bar{b}_{m} \rbrace$.
Note that all these elements lie in $R[[x_1, \dots, x_k]]$ and:
\[e^{ad w} \frac{\partial}{\partial x_i}\equiv \frac{\partial}{\partial x_i}-1 \otimes l_i (mod J), i=1, \dots, k;\]
\[e^{ad w} \bar{b}_{i} \equiv \bar{b}_{i} (mod J),\]
where $J$ is an ideal of $R[[x_1, \dots x_k]]$ generated by $x_1, \dots, x_k$.
The elements $\{l_1, \dots, l_k, \bar{b}_{k+1}, \dots, \bar{b}_m\}$ are linear independent.
Therefore $e^{ad w} \frac{\partial}{\partial x_1}-\frac{\partial}{\partial x_1},\dots,
e^{ad w} \frac{\partial}{\partial x_k}-\frac{\partial}{\partial x_k}, e^{ad w}(\bar{b}_{k+1}), \dots,
e^{ad w}(\bar{b}_{m})$ are linear independent and thus give a basis of $R \otimes L$.
Analogously we obtain that if $0 \neq l \in e^{ad w}\widehat{L}\cap 1 \otimes L$
then $0 \neq l \in \widehat{L}\cap 1 \otimes L$.
Hence using Lemma \ref{injectionconditions} we complete the proof of Lemma.
If $L$ is nilpotent (more precisely if $l_1, \dots l_k$ lie in
subalgebra $L'$ such that $ad:L' \rightarrow Der(L)$ is nilpotent representation)
then $ad w$ is nilpotent derivation of $R[x_1, \dots, x_k] \otimes L$ and thus
$e^{ad w} \bar{x_i} \in R(x_1, \dots, x_k), i=k+1, \dots, m$.
Thus we can change $F$ to $R(x_1, \dots, x_k)$.
\end{proof}

\begin{theorem}
Let $L$ be a finite dimensional Lie algebra over an algebraically closed field $\mathbb{K}$.
Then the following conditions are equivalent:

$(i)$ there exists a field $R$ and an injective homomorphism $\varphi: L \rightarrow Der_{\mathbb{K}}R$
such that $Dim_R R \varphi(L)=k$;

$(ii)$ there exists an irreducible subvariety of $V \subset M_k(L)$,
$V \nsubseteq M^0_k(L)$, such that an intersection of subalgebras associated
to points of $V$ is equal to $\{0\}$.
\end{theorem}

\begin{proof}
$(i) \Rightarrow (ii)$ It is immediate consequence of Proposition \ref{subalgebratovariety}.

$(ii) \Rightarrow (i)$ Let $S$ be a coordinate ring of the variety $V$.
Put $S=\mathbb{K}[z_{ij}, i=1 \dots m-k, j=1 \dots m]/I=\mathbb{K}[\bar{z}_{ij}]$,
$\bar{z}_{ij}=z_{ij}+I$ such that for any closed point
$p$ of $S$ we have $\lbrace \sum_{j=1}^m p(\bar{z}_{ij})l_j \rbrace$ be
a basis of subalgebra of codimension $k$. $V$ is irreducible. Hence $S$ is
an integral domain. Let $R$ be a quotient field of $S$. Then
$\lbrace \sum_{j=1}^m p(\bar{z}_{ij})l_j \rbrace$ is a basis of subalgebra
 $\widehat{L} \subset R \otimes L$. Assume that this set is linear dependent.
Then it is linear dependent in every point of $V$. Thus
we have that $V \subset M^0_k{L}$ with contradiction to $(ii)$. Therefore
$\widehat{L}$ has codimension $k$. Finally if $l\in 1 \otimes L \cap \widehat{L}$
then $l$ lies in all subalgebras associated to points $S$. Thus using Lemma \ref{subalgtoderivations} we obtain $(i)$.
\end{proof}

\section{Case of simple Lie algebras}
\begin{proposition} \label{simple}
Let $\mathbb{K}$ be an arbitrary field, $L$ be a simple Lie algebra over $\mathbb{K}$.
Assume that $L=L_1+L_2$, where $L_1$ has codimension $k$ and a representation $ad:L_2\rightarrow Der L$
is nilpotent. Then there exists a subalgebra $\varphi(L) \subset Der_{\mathbb{K}} \mathbb{K}(x_1, \dots x_k)$
such that $Dim_{\mathbb{K}(x_1, \dots x_k)} \mathbb{K}(x_1, \dots x_k) \varphi(L)=k$.
\end{proposition}
\begin{proof}
Denote $R=\mathbb{K}(x_1, \dots x_k)$.
Denote $M=Der_{\mathbb{K}} \mathbb{K}(x_1, \dots x_k)$,
$\mathfrak{D}=M \rightthreetimes \mathbb{K}(x_1, \dots x_k) \otimes L$.
Put $w= x_1 \otimes l_1 + \dots + x_k \otimes l_k$. It is easy to see that $ad w$ is a
nilpotent derivation of $\mathfrak{D}$.
Put $N=R\frac{\partial}{\partial x_1}, \dots, R\frac{\partial}{\partial x_k}+R \otimes L_1$ and consider an
$R$-module $e^{ad w} N$. Let $\{l_{k+1}, \dots, l_m\}$ be a basis of $L_1$. Then
we have that $\lbrace e^{ad w}(\frac{\partial}{\partial x_1}),
\dots, e^{ad w}(\frac{\partial}{\partial x_k}), e^{ad w} l_{k+1}, \dots, e^{ad w} l_m\rbrace$
is a basis of $e^{ad w} N$. Note that all these elements lie in polynomial ring $\mathbb{K}[x_1, \dots, x_k]$
and $e^{ad w}(\frac{\partial}{\partial x_i})-\frac{\partial}{\partial x_i}\equiv -l_i mod J$, $i=1, \dots, k$,
$e^{ad w}l_j \equiv l_j mod J$, $j=k+1, \dots, m$, where $J$ is an ideal of $\mathbb{K}[x_1, \dots, x_k]$,
generated by $x_1, \dots, x_k$. Thus we have that
$e^{ad w}(\frac{\partial}{\partial x_i})-\frac{\partial}{\partial x_i}, e^{ad w}l_j$ give an $R$ basis of
$R \otimes L$. Using Remark \ref{noninjective} we have a morphism
$\varphi: L \rightarrow Der_{\mathbb{K}}R$. It is injective because $L$ is simple. This completes
the proof of Proposition.
\end{proof}

\begin{corollary}
There exist an injection $\varphi:sl_n{\mathbb{K}} \rightarrow Der_{\mathbb{K}} \mathbb{K}(x_1, \dots, x_{n-1})$.
\end{corollary}
\begin{proof}
Let $\alpha_1, \dots, \alpha_{n-1}$ be a basis of root system, $sl_n(\mathbb{K})=n_-+h+n_+$,
where $h$ is a Cartan subalgebra, $n_+$ ($n_-$) is a linear span of all positive (negative) root
vectors. Put $L_1= h+n_-+\tilde{n_+}$ where $\tilde{n_+}$ is a subalgebra generated
by root vectors with roots in $\alpha_1, \dots, \alpha_{n-2}$. It is easy to see that $L_1$ is a
subalgebra and it has codimension $n-1$. $L=L_1+n_-$. Thus using Proposition \ref{simple}
we complete the proof of Lemma.
\end{proof}

\begin{example}
Let $L=sl_3(\mathbb{K})$, $\pm \alpha, \pm \beta, \pm (\alpha + \beta)$ be a root system.
Let $L_1=\langle h_{\alpha}, h_{\beta}, e_{-\alpha}, e_{-\beta}, e_{-\alpha-\beta}, e_{\beta} \rangle$.
Put $w=x_1 \otimes e_{\alpha}+x_2 \otimes e_{\alpha+\beta}$. Then:
\[e^{ad w} \frac{\partial}{\partial x_1}=\frac{\partial}{\partial x_1}-1 \otimes e_{\alpha},\]

\[e^{ad w} \frac{\partial}{\partial x_2}=\frac{\partial}{\partial x_2}-1 \otimes e_{\alpha+\beta},\]

\[e^{ad w} h_{\alpha}=1\otimes h_{\alpha}-2 x_1 \otimes e_{\alpha}-x_2 \otimes e_{\alpha+ \beta},\]

\[e^{ad w} h_{\beta}=1\otimes h_{\beta}+ x_1 \otimes e_{\alpha}-x_2 \otimes e_{\alpha+ \beta},\]

\[e^{ad w} e_{-\beta}=1\otimes e_{-\beta}-x_2 \otimes e_{\alpha},\]

\[e^{ad w} e_{\beta}=1\otimes e_{\beta}-x_1 \otimes e_{\alpha+ \beta},\]

\[e^{ad w} e_{-\alpha}=1\otimes e_{-\alpha}+x_1 \otimes h_{\alpha}+x_2 \otimes e_{\beta}-x_1^2 \otimes e_{\alpha}
-x_1x_2 \otimes e_{\alpha+\beta},\]

\[e^{ad w} e_{-\alpha-\beta}=1\otimes e_{-\alpha-\beta}+ x_1 \otimes e_{-\beta}+x_2 \otimes (h_{\alpha}+h_{\beta})
-x_2^2 \otimes e_{\alpha+\beta}-x_1 x_2 \otimes e_{\alpha}.\]

Then solving the linear system \ref{systemofequations} we have:

\[\varphi(e_{\alpha})=-\frac{\partial}{\partial x_1}, \varphi(e_{\alpha+\beta})=-\frac{\partial}{\partial x_2},
\varphi(h_{\alpha})=-2x_1 \frac{\partial}{\partial x_1} -x_2 \frac{\partial}{\partial x_2},
\varphi(h_{\beta})=x_1 \frac{\partial}{\partial x_1} -x_2 \frac{\partial}{\partial x_2},\]
\[\varphi(e_{\beta})=-x_1 \frac{\partial}{\partial x_2}, \varphi(e_{-\beta})=-x_2 \frac{\partial}{\partial x_1},
\varphi(e_{-\alpha})=-x_1^2 \frac{\partial}{\partial x_1}-x_1 x_2 \frac{\partial}{\partial x_2},\]
\[\varphi(e_{-\alpha-\beta})=-x_1x_2 \frac{\partial}{\partial x_1}-x_2^2 \frac{\partial}{\partial x_2}.\]
\end{example}

\end{document}